\documentclass{siamltex}
\usepackage{graphicx}
\usepackage[ruled,vlined]{algorithm2e}
\usepackage{subfigure}
\usepackage{epsfig,psfig}
\usepackage{amsfonts,latexsym}
\usepackage{amsmath}

\newcommand{\R}{{\mathbb R}}

\newcommand{\bx}{\mbox{\boldmath{$x$}}}
\newcommand{\bb}{\mbox{\boldmath{$b$}}}

\newcommand{\be}{\mbox{\boldmath{$e$}}}

\newcommand{\bn}{\mbox{\boldmath{$n$}}}

\newcommand{\by}{\mbox{\boldmath{$y$}}}

\newcommand{\sby}{\mbox{\boldmath{${\scriptstyle y}$}}}
\newcommand{\sbx}{\mbox{\boldmath{${\scriptstyle x}$}}}

\newcommand{\bzero}{\mbox{\boldmath{$0$}}}
\newcommand{\Null}[1]{\ensuremath{\mathcal{N}({#1})}}

\newcommand{\bl}{\begin{list}{ \ }{
\leftmargin=.325in}}
\newcommand{\el}{\end{list}}

\begin{document}
\date{}
\title{Regularization matrices for discrete ill-posed problems in several
space-dimensions}
\author{
Laura Dykes\thanks{University School, Hunting Valley, OH 44022, USA. E-mail:
{\tt ldykes@math.kent.edu.}}
\and
Guangxin Huang\thanks{Geomathematics Key Laboratory of Sichuan, College of Management
Science, Chengdu University of Technology, Chengdu, 610059, P. R. China. E-mail:
{\tt huangx@cdut.edu.cn.}}
\and
Silvia Noschese\thanks{Department of Mathematics, SAPIENZA Universit\`a di Roma, P.le
A. Moro, 2, I-00185 Roma, Italy. E-mail: {\tt noschese@mat.uniroma1.it}.}
\and
Lothar Reichel\thanks{Department of Mathematical Sciences, Kent State University, Kent,
OH 44242, USA. E-mail: {\tt reichel@math.kent.edu}.}
}

\maketitle

\begin{abstract}
Many applications in science and engineering require the solution of large linear discrete
ill-posed problems that are obtained by the discretization of a Fredholm integral equation
of the first kind in several space-dimensions. The matrix that defines these problems 
is very ill-conditioned and generally numerically singular, and the right-hand side, which 
represents measured data, typically is contaminated by measurement error. Straightforward 
solution of these problems generally is not meaningful due to severe error propagation. 
Tikhonov regularization seeks to alleviate this difficulty by replacing the given linear 
discrete ill-posed problem by a penalized least-squares problem, whose solution is less 
sensitive to the error in the right-hand side and to round-off errors introduced during 
the computations. This paper discusses the construction of penalty terms that are 
determined by solving a matrix-nearness problem.  These penalty terms allow partial 
transformation to standard form of Tikhonov regularization problems that stem from the 
discretization of integral equations on a cube in several space-dimensions.
\end{abstract}



\section{Introduction}\label{sec1}
We consider the solution of linear discrete ill-posed problems that arise from the
discretization of a Fredholm integral equation of the first kind on a cube in two or more
space-dimensions. Discretization of the integral operator yields a matrix
$K\in\R^{M\times N}$, which we assume to be large. A vector $\bb\in\R^M$ that represents
measured data, and therefore is error-contaminated, is available and we would like to
compute an approximate solution of the least-square problem
\begin{equation}\label{eq:sys}
   \min_{\sbx\in\R^N} \|K\bx-\bb\|.
\end{equation}
The matrix $K$ has many ``tiny'' singular values of different orders of magnitude. This
makes $K$ severely ill-conditioned; in fact, $K$ may be numerically singular.
Least-squares problems \eqref{eq:sys} with a matrix of this kind are commonly referred to
as linear discrete ill-posed problems.

Let $\be\in\R^M$ denote the (unknown) error in $\bb$. Thus,
\begin{equation}\label{noisefree}
\bb=\widehat{\bb}+\be,
\end{equation}
where $\widehat{\bb}\in\R^M$ stands for the unknown error-free vector associated with
$\bb$. We will assume the unavailable linear system of equations
\begin{equation}\label{consistent}
K\bx=\widehat{\bb}
\end{equation}
to be consistent.

Let $K^\dag$ denote the Moore--Penrose pseudoinverse of the matrix $K$. We are interested
in determining the solution $\widehat{\bx}$ of \eqref{consistent} of minimal Euclidean
norm; it is given by $K^\dag\widehat{\bb}$. We note that the solution of minimal Euclidean
norm of \eqref{eq:sys},
\[
K^\dag\bb=K^\dag\widehat{\bb}+K^\dag\be=\widehat{\bx}+K^\dag\be,
\]
generally is not a meaningful approximation of $\widehat{\bx}$ due to a large propagated
error $K^\dag\be$. This difficulty stems from the fact that $\|K^\dag\|$ is large.
Throughout this paper $\|\cdot\|$ denotes the Euclidean vector norm or spectral matrix
norm. We also will use the Frobenius norm of a matrix, defined by
$\|K\|_F=\sqrt{{\rm trace}(K^TK)}$, where the superscript $^T$ stands for transposition.

To avoid severe error propagation, one replaces the least-squares problem \eqref{eq:sys}
by a nearby problem, whose solution is less sensitive to the error $\be$ in $\bb$. This
replacement is known as regularization. Tikhonov regularization, which is one of the most
popular regularization methods, replaces \eqref{eq:sys} by a penalized least-squares
problem of the form
\begin{equation}\label{eq:tik}
    \min_{{\sbx}\in\R^N}\left\{\|K\bx-\bb\|^2+\mu\|L\bx\|^2\right\},
\end{equation}
where $L\in\R^{J\times N}$ is referred to as the regularization matrix and the scalar
$\mu>0$ as the regularization parameter; see, e.g., \cite{BRRS,EHN,Ha2}. We assume the
matrix $L$ to be chosen so that
\begin{equation}\label{nullcond}
\Null{K}\cap \Null{L} = \{\bzero\},
\end{equation}
where $\Null{M}$ denotes the null space of the matrix $M$. Then the minimization problem
\eqref{eq:tik} has a unique solution
\[
\bx_\mu=(K^TK+\mu L^TL)^{-1}K^T\bb
\]
for any $\mu>0$.

Common choices of $L$ include the identity matrix and discretizations of differential
operators. The Tikhonov minimization problem \eqref{eq:tik} is said to be in
\emph{standard form} when $L=I$; otherwise it is in \emph{general form}. Numerous computed
examples in the literature, see, e.g., \cite{CRS,DNR1,GNR,RYe}, illustrate that the
choice of $L$ can be important for the quality of the computed approximation $\bx_\mu$ of
$\widehat{\bx}$. The regularization matrix $L$ should be chosen so that known important
features of the desired solution $\widehat{\bx}$ of \eqref{consistent} are not damped.
This can be achieved by choosing $L$ so that $\Null{L}$ contains vectors that represent
these features, because vectors in $\Null{L}$ are not damped by $L$.

Several approaches to construct regularization matrices with desirable properties are
described in the literature; see, e.g., \cite{BJ,CRS,DNR1,DR1,HJ,HNR,KHE,MRS,NR,RYe,RYu2}.
Huang et al. \cite{HNR} propose the construction of square regularization matrices with a
user-specified null space by solving a matrix nearness problem in the Frobenius norm. The
regularization matrices so obtained are designed for linear discrete ill-posed problems in
one space-dimension. This paper extends this approach to problems in higher
space-dimensions. The regularization matrices of this paper generalize those applied by
Bouhamidi and Jbilou \cite{BJ} by allowing a user-specified null space.

Consider the special case of $d=2$ space-dimensions and let the matrix $K$ be determined
by discretizing an integral equation on a square $n\times n$ grid (i.e., $N=n^2$). The
regularization matrix
\begin{equation}\label{Ltensor1}
L_{1,\otimes}=\left[\begin{array}{ccc} I & \otimes & L_1 \\ L_1 & \otimes & I
\end{array}\right],
\end{equation}
where $I\in\R^{n\times n}$ is the identity matrix,
\begin{equation}\label{L1}
L_1 = \frac{1}{2}\left[ \begin{array} {cccccc}
 1 &   -1    &     &    &   & \mbox{\Large 0} \\
   &  \phantom{-}1    &   -1    &     &    & \\
   &   &  \phantom{-}1    &  -1  &    &    \\
   &   &  & \ddots & \ddots  &   \\
 \mbox{\Large 0}   &        &        &    & \phantom{-}1  & -1
\end{array}
\right]\in{\R}^{(n-1)\times n},
\end{equation}
and $\otimes$ denotes the Kronecker product, has frequently been used for this kind of
problem; see e.g., \cite{CLR,HRY,KHE,LMRS,RYu1}. We note for future reference that
${\mathcal N}(L_1)={\rm span}\{[1,1,\ldots,1]^T\}$.

It also may be attractive to replace the matrix \eqref{L1} in \eqref{Ltensor1} by
\begin{equation}\label{L2}
L_2 = \frac{1}{4}\left[ \begin{array} {cccccc}
 -1   &  \phantom{-}2    &   -1    &     &    &\mbox{\Large 0} \\
      &   -1  &  \phantom{-}2    &  -1  &    &    \\
      &       & \ddots & \ddots & \ddots  &   \\
  \mbox{\Large 0}    &        &        &   -1  & \phantom{-}2      & -1
\end{array}
\right]\in{\R}^{(n-2)\times n}
\end{equation}
with null space ${\mathcal N}(L_2)={\rm span}\{[1,1,\ldots,1]^T,[1,2,\ldots,n]^T\}$. This
yields the regularization matrix
\begin{equation}\label{Ltensor2}
L_{2,\otimes}=\left[\begin{array}{ccc} I & \otimes & L_2 \\ L_2 & \otimes & I
\end{array}\right].
\end{equation}

Both the regularization matrices \eqref{Ltensor1} and \eqref{Ltensor2} are rectangular
with almost twice as many rows as columns when $n$ is large.

Bouhamidi and Jbilou \cite{BJ} proposed the use of the smaller invertible regularization
matrix
\begin{equation}\label{Ltensor3}
L_{2,\otimes}=\widetilde{L}_2\otimes\widetilde{L}_2,
\end{equation}
where
\begin{equation}\label{L2square}
\widetilde{L}_2 = \frac{1}{4}\left[ \begin{array} {cccccc}
\phantom{-}2    &   -1    &     &    &   &  \mbox{\Large 0} \\
 -1   &  \phantom{-}2    &   -1    &     &    \\
      &   -1  &  \phantom{-}2    &  -1  &    &    \\
      &       & \ddots & \ddots & \ddots  &   \\
      &        &        &   -1  & \phantom{-}2      & -1\\
\mbox{\Large 0}    &    &    &        &   -1  & \phantom{-}2
\end{array}
\right]\in{\R}^{n\times n}
\end{equation}
is a square nonsingular regularization matrix. Therefore the regularization matrix
\eqref{Ltensor3} also is square and nonsingular, which makes it easy to transform the
Tikhonov minimization problem \eqref{eq:tik} so obtained to standard form; see below.

Following Bouhamidi and Jbilou \cite{BJ}, we consider square matrices $K$ with a tensor
product structure, i.e.,
\begin{equation}\label{Ktensor}
K=K^{(2)}\otimes K^{(1)}.
\end{equation}
We assume for simplicity that $K^{(1)},K^{(2)}\in{\R}^{n\times n}$ with $N=n^2$. However,
we note that the regularization matrices described in this paper can be applied also when
the matrix $K$ in \eqref{eq:sys} does not have a tensor product structure.

Bouhamidi and Jbilou \cite{BJ} are concerned with applications to image restoration and
achieve restorations of high quality. However, for linear discrete ill-posed problems in
one space-dimension, analysis presented in \cite{CRS,HJ} indicates that the regularization
matrix \eqref{L2}, with a non-trivial null space, can give approximate solutions of higher
quality than the matrix \eqref{L2square}, which has a trivial null space. This depends on
that the latter matrix may introduce artifacts close to the boundary; see also
\cite{DNR1,DR1,RYe} for related discussions and illustrations. It is the aim of the
present paper to develop a generalization of the regularization matrix \eqref{Ltensor2}
that has a non-trivial null space. Our approach to define such a regularization matrix
generalizes the technique proposed in \cite{HNR} from one to several space-dimensions.
Specifically, the regularization matrix is defined by solving a matrix nearness problem in
the Frobenius norm. The regularization matrix so obtained allows a partial transformation
of the Tikhonov regularization problem \eqref{eq:tik} to standard form. When the matrix
$K$ is square, Arnoldi-type iterative solution methods can be used. Arnoldi-type iterative
solution methods often require fewer matrix-vector product evaluations than iterative
solution methods based on Golub--Kahan bidiagonalization, because they do not require
matrix-vector product evaluations with $K^T$; see, e.g., \cite{LR} for illustrations. A
nice recent survey of iterative solution methods for discrete ill-posed problems is
provided by Gazzola et al.  \cite{GNR}.

This paper is organized as follows. Section \ref{sec2} describes our construction of new
regularization matrices for problems in two space-dimensions. The section also discusses
iterative methods for the solution of the Tikhonov minimization problems obtained. We
consider both the situation when $K$ is a general matrix and when $K$ has a tensor product
structure. Section \ref{sec3} generalizes the results of Section \ref{sec2} to more than
two space-dimensions. Computed examples can be found in Section \ref{sec4}, and Section
\ref{sec5} contains concluding remarks.

We conclude this section by noting that while this paper focuses on iterative solution
methods for large-scale Tikhonov minimization problems \eqref{eq:tik}, the regularization
matrices described also can be applied in direct solution methods for small to
medium-sized problems that are based on the generalized singular value decomposition
(GSVD); see, e.g., \cite{DNR2,Ha2} for discussions and references.

\section{Regularization matrices for problems in two space-dimensions}\label{sec2}
Many image restoration problems as well as problems from certain other applications
(\ref{eq:sys}) have a matrix $K\in\R^{N\times N}$ that is the Kronecker product of two
matrices $K^{(1)},K^{(2)}\in\R^{n\times n}$ with $N=n^2$, cf. \eqref{Ktensor}. We will
consider this situation in most of this section; the case when $K$ is a general square 
matrix without Kronecker product structure is commented on at the end of the section. 
Extension to rectangular matrices $K$, $K^{(1)}$, and $K^{(2)}$ is straightforward.

We will use regularization matrices with a Kronecker product structure,
\begin{equation}\label{Ltensor}
L = L^{(2)} \otimes L^{(1)}
\end{equation}
and will discuss the choice of square regularization matrices
$L^{(1)},L^{(2)}\in\R^{n\times n}$. The following result is an extension of 
\cite[Proposition 2.1]{HNR} to problems with a Kronecker product structure. Let 
${\mathcal R}(A)$ denote the range of the matrix $A$ and define the Frobenius inner 
product
\begin{equation}\label{globip}
\langle A_1,A_2\rangle={\rm trace}(A_1^TA_2)
\end{equation}
between matrices $A_1,A_2\in{\R}^{m_1\times m_2}$. Throughout this section $N=n^2$.

\begin{proposition}\label{prop1}
Let the matrices $V^{(1)}\in\R^{n\times\ell_1}$ and $V^{(2)}\in\R^{n\times\ell_2}$ have
orthonormal columns, and let ${\mathcal B}$ denote the subspace of matrices of the form
$B=B^{(2)}\otimes B^{(1)}$, where the null space of $B^{(i)}\in\R^{n\times n}$ contains
${\mathcal R}(V^{(i)})$ for $i=1,2$. Introduce for $i=1,2$ the orthogonal projectors
$P^{(i)}=I-V^{(i)}V^{(i)T}$ with null space ${\mathcal R}(V^{(i)})$. Let
$P=P^{(2)}\otimes P^{(1)}$. Then the matrix $\widehat{A}=AP$ is a closest matrix to
$A=A^{(2)}\otimes A^{(1)}$ with $A^{(i)}\in\R^{n\times n}$, $i=1,2$, in ${\mathcal B}$ in
the Frobenius norm.
\end{proposition}

\begin{proof}
The matrix $\widehat{A}$ satisfies the following conditions:
\begin{enumerate}
\item
$\widehat{A}\in {\cal B}$;
\item
if $A\in {\cal B}$, then $\widehat{A} \equiv A$;
\item
if $B\in {\cal B}$, then $\langle B,A-\widehat{A}\rangle=0$.
\end{enumerate}
In fact,
\[
\widehat{A}(V^{(2)}\otimes V^{(1)})=A(P^{(2)}V^{(2)}\otimes P^{(1)}V^{(1)})=0,
\]
which shows the first property. The second property implies that
\[
A^{(2)}V^{(2)}=0,\qquad A^{(1)}V^{(1)}=0,
\]
from which it follows that
\[
\widehat{A}=(A^{(2)}-A^{(2)}V^{(2)}V^{(2)T})\otimes(A^{(1)}-A^{(1)}V^{(1)}V^{(1)T})=
A^{(2)}\otimes A^{(1)}=A.
\]
Finally, for any $B\in{\mathcal B}$ of the form $B=B^{(2)}\otimes B^{(1)}$, one has that
\[
B^{(2)}V^{(2)}=V^{(2)T}B^{(2)T}=0,\qquad B^{(1)}V^{(1)}=V^{(1)T}B^{(1)T}=0,
\]
so that
\begin{eqnarray*}
\langle B,A-\widehat{A}\rangle &=&\textrm{trace}(B^TA-B^T\widehat{A}) \\
&=&\textrm{trace}(B^{(2)T}A^{(2)})\,\textrm{trace}(B^{(1)T}A^{(1)}V^{(1)}V^{(1)T})\\
&&+\textrm{trace}(B^{(2)T}A^{(2)}V^{(2)}V^{(2)T})\,\textrm{trace}(B^{(1)T}A^{(1)})\\
&&-\textrm{trace}(B^{(2)T}A^{(2)}V^{(2)}V^{(2)T})\, 
\textrm{trace}(B^{(1)T}A^{(1)}V^{(1)}V^{(1)T})=0,
\end{eqnarray*}
where the last equality follows from the cyclic property of the trace.
\end{proof}

Example 2.1. Let $L_2$ and $\widetilde{L}_2$ be defined by \eqref{L2} and
\eqref{L2square}, respectively. Proposition \ref{prop1} shows that a closest matrix to
$\widetilde{L}=\widetilde{L}_2\otimes \widetilde{L}_2$ in the Frobenius norm with null
space ${\mathcal N}(L_2\otimes L_2)$ is
\[
L=\widetilde{L}_2P_2\otimes \widetilde{L}_2P_2,
\]
where $P_2$ is the orthogonal projector onto ${\mathcal N}(L_2)^\perp$.
$~~~\Box$

Example 2.2. Define the square nonsingular regularization matrix
\begin{equation}\label{L1square}
\widetilde{L}_1 = \frac{1}{2}\left[ \begin{array} {cccccc}
 1 &   -1  &    &    &   & \mbox{\Large 0} \\
   &  \phantom{-}1    &   -1  &     &    & \\
   &   &  \phantom{-}1    &  -1  &   &    \\
   &   &  &  \ddots & \ddots  &   \\
   &   &  &  &  &  -1 \\
   \mbox{\Large 0}   &        &    &    &    & \phantom{-}1
 \end{array}
  \right]\in{\R}^{n\times n}.
\end{equation}
A closest matrix to $\widetilde{L}=\widetilde{L}_1\otimes \widetilde{L}_1$ in the
Frobenius norm with null space ${\mathcal N}(L_1\otimes L_1)$ is given by
\[
L=\widetilde{L}_1P_1\otimes \widetilde{L}_1P_1,
\]
where $P_1$ is the orthogonal projector onto ${\mathcal N}(L_1)^\perp$; see, e.g.,
\cite{HNR}. $~~~\Box$

The following result is concerned with the situation when the order of the nonsingular
matrices $\widetilde{L}_i$ and projectors $P_i$ in Examples 2.1 and 2.2 is reversed. We
first consider the situation when $\widetilde{L}$ is a square matrix without Kronecker
product structure, since this situation is of independent interest.

\begin{proposition}\label{prop2}
Let $\widetilde{L}\in\R^{n\times n}$ and let ${\mathcal V}$ be a subspace of $\R^n$. Define
the orthogonal projector $P_{{\mathcal V}^\perp}$ onto ${\mathcal V}^\perp$. Then the
closest matrix $\widehat{L}\in\R^{n\times n}$ to $\widetilde{L}$ in the Frobenius norm,
such that ${\mathcal R}(\widehat{L})\subset{\mathcal V}^\perp$, is given by
$\widehat{L}=P_{{\mathcal V}^\perp}\widetilde{L}$.
\end{proposition}

\begin{proof}
Consider the problem of determining a closest matrix $\widehat{L}^T\in\R^{n\times n}$ to
$\widetilde{L}^T$ in the Frobenius norm whose null space contains ${\mathcal V}$. It is
shown in \cite[Proposition 2.3]{HNR} that
$\widehat{L}^T=\widetilde{L}^TP_{{\mathcal V}^\perp}$ is such a matrix. The Frobenius norm
is invariant under transposition and orthogonal projectors are symmetric. Therefore,
\[
\|\widetilde{L}^TP_{{\mathcal V}^\perp}-\widetilde{L}^T\|_F=
\|P_{{\mathcal V}^\perp}\widetilde{L}-\widetilde{L}\|_F.
\]
Moreover, ${\mathcal R}(P_{{\mathcal V}^\perp})={\mathcal V}^\perp$. It follows that a
closest matrix to $\widetilde{L}$ in the Frobenius norm whose range is a subset of
${\mathcal V}^\perp$ is given by $P_{{\mathcal V}^\perp}\widetilde{L}$.
\end{proof}

The following result extends Proposition \ref{prop2} to matrices with a tensor product
structure. We formulate the result similarly as Proposition \ref{prop1}.

\begin{corollary}\label{cor1}
Let the matrices $V^{(1)}\in\R^{n\times\ell_1}$ and $V^{(2)}\in\R^{n\times\ell_2}$ have
orthonormal columns, and let ${\mathcal B}$ denote the subspace of matrices of the form
$B=B^{(2)}\otimes B^{(1)}$, where the range of $B^{(i)}\in\R^{n\times n}$ is orthogonal to
${\mathcal R}(V^{(i)})$ for $i=1,2$. Introduce for $i=1,2$ the orthogonal projectors
$P^{(i)}=I-V^{(i)}V^{(i)T}$ and let $P=P^{(2)}\otimes P^{(1)}$. Then the matrix
$\widehat{A}=PA$ is a closest matrix to $A=A^{(2)}\otimes A^{(1)}$ with
$A^{(i)}\in\R^{n\times n}$, $i=1,2$, in ${\mathcal B}$ in the Frobenius norm.
\end{corollary}

\begin{proof}
The result can be shown by applying Propositions \ref{prop1} or \ref{prop2}.
\end{proof}

Example 2.3. Let $L_2$ be defined by \eqref{L2} and $\widetilde{L}_2$ by \eqref{L2square}.
Corollary \ref{cor1} shows that a closest matrix to
$\widetilde{L}=\widetilde{L}_2\otimes \widetilde{L}_2$ with range in
${\mathcal R}(L_2\otimes L_2)$ is
\[
L=P_2\widetilde{L}_2\otimes P_2\widetilde{L}_2,
\]
where $P_2={\rm diag}[0,1,1,\ldots,1,0]\in\R^{n\times n}$. $~~~\Box$

Example 2.4. Let the matrices $L_1$ and $\widetilde{L}_1$ be given by \eqref{L1} and
\eqref{L1square}. It follows from Corollary \ref{cor1} that a closest matrix
to $\widetilde{L}=\widetilde{L}_1\otimes\widetilde{L}_1$ with range in
${\mathcal R}(L_1\otimes L_1)$ is given by
\[
L=P_1\widetilde{L}_1\otimes P_1\widetilde{L}_1,
\]
where $P_1={\rm diag}[1,1,\ldots,1,0]\in\R^{n\times n}$. $~~~\Box$

Using \eqref{Ktensor} and \eqref{Ltensor}, the Tikhonov regularization problem
\eqref{eq:tik} can be expressed as
\begin{equation}\label{scal:tik1}
\min_{{\sbx}\in\R^N}\left\{\|(K^{(2)} \otimes K^{(1)})\bx-\bb\|^2+
\mu\|(L^{(2)} \otimes L^{(1)})\bx\|^2\right\}.
\end{equation}

It is convenient to introduce the operator $\mathrm{vec}$, which transforms a matrix
$Y\in{\R}^{n\times n}$ to a vector of size $n^2$ by stacking the columns of $Y$. Let $A$,
$B$, and $Y$, be matrices of commensurate sizes. Then
\[
\mathrm{vec}(AYB) = (B^T\otimes A) \mathrm{vec}(Y);
\]
see, e.g., \cite{HJbook} for operations on matrices with Kronecker product structure. We
can apply this identity to express \eqref{scal:tik1} in the form
\begin{equation}\label{blkeq:tik1}
 \min_{X\in\R^{n\times n}}\left\{\|K^{(1)}XK^{(2)T}-B\|_F^2+
 \mu\|L^{(1)} XL^{(2)T}\|_F^2\right\},
\end{equation}
where the matrix $B\in{\R}^{n\times n}$ satisfies $\bb=\mathrm{vec}(B)$.

Let the regularization matrices in \eqref{blkeq:tik1} be of the forms
\begin{equation}\label{L1L2}
L^{(1)}=P^{(1)}\widetilde{L}^{(1)},\qquad L^{(2)}=P^{(2)}\widetilde{L}^{(2)},
\end{equation}
where the matrices $\widetilde{L}^{(i)}\in\R^{n\times n}$ are nonsingular and the
$P^{(i)}$ are orthogonal projectors. We easily can transform \eqref{blkeq:tik1} to a form
with an orthogonal projector regularization matrix,
\begin{equation}\label{blkeq:Ptik2}
\min_{Y\in\R^{n\times s}}\left\{\|K_1^{(1)}YK_1^{(2)T}-B\|_F^2+
\mu\|P^{(1)}Y P^{(2)}\|_F^2\right\},
\end{equation}
where
\[
K_1^{(i)}=K^{(i)}(\widetilde{L}^{(i)})^{-1},\qquad i=1,2.
\]

We will solve \eqref{blkeq:Ptik2} by an iterative method. The structure of the
minimization problem makes it convenient to apply an iterative method based on the global
Arnoldi process, which was introduced and first analyzed by Jbilou et al. \cite{JMS,JST}.
We refer to matrices with many more rows than columns as ``block vectors''. The block
vectors $U,W\in\R^{N\times n}$ are said to be $F$-\emph{orthogonal} if
$\langle U,W\rangle = 0$; they are $F$-\emph{orthonormal} if in addition
$\|U\|_F=\|W\|_F=1$.

The application of $k$ steps of the global Arnoldi method to the solution of 
\eqref{blkeq:Ptik2} yields an $F$-orthonormal basis $\{V_1,V_2,\ldots,V_k\}$ of block 
vectors $V_j$ for the block Krylov subspace
\begin{equation}\label{gkryl}
{\mathcal K}_k={\rm span}\{B,K_1^{(1)}BK_1^{(2)T},\ldots,
(K_1^{(1)})^{k-1}B(K_1^{(2)T})^{k-1}\}.
\end{equation}
In particular $V_1=B/\|B\|_F$. The use of the global Arnoldi method to the solution of
\eqref{blkeq:Ptik2} is mathematically equivalent to applying a standard Arnoldi method to
\eqref{scal:tik1}. An advantage of the global Arnoldi method is that it can be implemented
by using matrix-matrix operations, while the standard Arnoldi method applies matrix-vector
and vector-vector operations. This can lead to faster execution of the global Arnoldi
method on many modern computers. Algorithm \ref{algo1} outlines the global Arnoldi method;
see \cite{JMS,JST} for further discussions of this and other block methods.

\begin{algorithm}
\DontPrintSemicolon
\nl compute $V_1=B/\|B\|_F$\;
\nl for $j=1,2,\dots, k$ compute \;
\nl $\qquad $  $V=K_1^{(1)}V_j$  \;
\nl $\qquad $  $V=VK_1^{(2)T}$  \;
\nl  $\qquad $ for $i=1,2,\dots, j$ \;
\nl  $\qquad \qquad$   $h_{i,j}=\langle V,V_i\rangle$ \;
\nl  $\qquad \qquad$   $V=V-h_{i,j}V_i$ \;
\nl  $\qquad $  end\;
\nl  $\qquad $  $h_{j+1,j}=\|V\|_F$ \;
\nl  $\qquad $ if $h_{j+1,j}=0\;\;$ stop \;
\nl  $\qquad $  $V_{j+1}=V/h_{j+1,j}$  \;
\nl end \;
\nl construct the $n\times kn$ matrix $\widehat{V}_k=[V_1,\dots, V_k]$ with
$F$-orthonormal block columns $V_j$. The block columns span the space \eqref{gkryl} \;
\nl construct the $(k+1)\times k$ Hessenberg matrix
$\widetilde{H}_k=[h_{i,j}]_{i=1,2,\dots, k+1,j=1,2,\dots, k}$\;
\caption{Global Arnoldi for computing an $F$-orthonormal basis for \eqref{gkryl}}\label{algo1}
\end{algorithm}

We determine an approximate solution of \eqref{blkeq:Ptik2} in the global Arnoldi subspace
\eqref{gkryl}. This is described by Algorithm \ref{algo2} for a given $\mu>0$. The
solution subspace \eqref{gkryl} is independent of the orthogonal projectors that determine
the regularization term in \eqref{blkeq:Ptik2}. This approach to generate a solution
subspace for the solution of Tikhonov minimization problems in general form was first
discussed in \cite{HR}; see also \cite{GNR} for examples.

\begin{algorithm}
\DontPrintSemicolon
\nl construct $\widehat{V}_k=[V_1,V_2,\ldots,V_k]$ and $\widetilde{H}_k$ using Algorithm
\ref{algo1}\;
\nl solve for a given $\mu>0$,
\[
\min_{\sby\in\mathbb{R}^k}\left\{\left\|\widetilde{H}_k\by-\|B\|_F\,\be_1\right\|^2+\mu
\left\|\sum_{i=1}^k y_iP^{(1)}V_iP^{(2)}\right\|_F^2\right\},
\]
where $\be_1=[1,0,\ldots,0]^T\in\R^{k+1}$ and $\by=[y_1,y_2,\ldots,y_k]^T$\;
\nl compute  $Y_{\mu,k}=\sum_{i=1}^k V_i y_i$\;
\caption{Tikhonov regularization based on the global Arnoldi process}\label{algo2}
\end{algorithm}

The discrepancy principle is a popular approach to determine the regularization parameter
$\mu$ when a bound $\varepsilon$ for the norm of the error $\be$ in $\bb$ is known, i.e.,
$\|\be\|\leq\varepsilon$. It prescribes that $\mu>0$ be chosen so that the computed
approximate solution $Y_{\mu,k}$ of \eqref{blkeq:Ptik2} satisfies
\begin{equation}\label{nonlin}
\|K_2^{(1)}Y_{\mu,k}K_2^{(2)T}-B\|_F=\eta\varepsilon,
\end{equation}
where $\eta\geq 1$ is a user-chosen constant independent of $\varepsilon$. The nonlinear
equation \eqref{nonlin} for $\mu$ can be solved by a variety of methods such as Newton's
method; see \cite{HR} for a discussion.

Finally, we note that the regularization matrices of this section also can be applied when
the matrix $K$ in \eqref{eq:tik} does not have a Kronecker product structure
\eqref{Ktensor}. Let ${\bx}={\rm vec}(X)$. Then the matrix expression in the penalty term
of \eqref{blkeq:Ptik2} can be written as
\[
P^{(1)}\widetilde{L}^{(1)}X\widetilde{L}^{(2)T}P^{(2)}=
((P^{(2)}\widetilde{L}^{(2)})\otimes(P^{(1)}\widetilde{L}^{(1)})){\bx}=
(P^{(2)}\otimes P^{(1)})(\widetilde{L}^{(2)}\otimes\widetilde{L}^{(1)}){\bx}.
\]
The analogue of the minimization problem \eqref{blkeq:Ptik2} therefore can be expressed as
\begin{equation}\label{eqvec}
\min_{{\sbx}\in\R^N}\left\{\|K\bx-\bb\|^2+\mu\|(P^{(2)}\otimes P^{(1)})
(\widetilde{L}^{(2)}\otimes\widetilde{L}^{(1)}){\bx}\|^2\right\}.
\end{equation}
The matrix $\widetilde{L}^{(2)}\otimes\widetilde{L}^{(1)}$ is invertible; we have
$(\widetilde{L}^{(2)}\otimes\widetilde{L}^{(1)})^{-1}=
(\widetilde{L}^{(2)})^{-1}\otimes(\widetilde{L}^{(1)})^{-1}$. It follows that the problem
\eqref{eqvec} can be transformed to
\begin{equation}\label{eqvec2}
\min_{{\sby}\in\R^N}\left\{\|K(
(\widetilde{L}^{(2)})^{-1}\otimes(\widetilde{L}^{(1)})^{-1})\by-\bb\|^2+
\mu\|(P^{(2)}\otimes P^{(1)})\by\|^2\right\}.
\end{equation}

The matrix $P^{(2)}\otimes P^{(1)}$ is an orthogonal projector. It is described in
\cite{MRS} how Tikhonov regularization problems with a regularization term that is
determined by an orthogonal projector with a low-dimensional null space easily can be
transformed to standard form. However,
${\rm dim}({\mathcal N}(P^{(2)}\otimes P^{(1)}))\geq n$, which generally is quite large in
problems of interest to us. It is therefore impractical to transform the Tikhonov
minimization problem \eqref{eqvec2} to standard form. We can solve \eqref{eqvec2}, e.g.,
by generating a (standard) Krylov subspace determined by the matrix
$K((\widetilde{L}^{(2)})^{-1}\otimes(\widetilde{L}^{(1)})^{-1})$ and vector $\bb$,
similarly as described in \cite{HR}. When the matrix $K$ is square, the Arnoldi process
can be applied to generate a solution subspace; when $K$ is rectangular, partial
Golub--Kahan bidiagonalization of $K$ can be used. The latter approach requires
matrix-vector product evaluations with both $K$ and $K^T$; see \cite{HR} for further
details.

\section{Regularization matrices for problems in higher space-dimensions}\label{sec3}
Proposition \ref{prop1} can be extended to higher space-dimensions. In addition to
allowing $d\geq 2$ space-dimensions, we remove the requirement that all blocks be square
and of equal size.

\begin{proposition}\label{prop3}
Let $V_{\ell_i}^{(i)}\in\R^{n_i\times\ell_i}$ have $1\leq\ell_i<n_i$ orthonormal columns for
$i=1,2,\ldots,d$, and let ${\mathcal B}$ denote the subspace of matrices of the form
\[
B=B^{(d)}\otimes B^{(d-1)}\otimes\cdots \otimes B^{(1)},
\]
where the null space of $B^{(i)}\in\R^{p_i\times n_i}$ contains ${\mathcal R}(V_{\ell_i}^{(i)})$
for all $i$. Let $I_k$ denote the identity matrix of order $k$ and define the orthogonal
projectors
\begin{equation}\label{Pi}
P=P^{(d)}\otimes P^{(d-1)}\otimes\cdots\otimes P^{(1)},
\quad P^{(i)}=I_{n_i}-V_{\ell_i}^{(i)}V_{\ell_i}^{(i)T},\quad i=1,2,\ldots,d.
\end{equation}
Then the matrix $\widehat{A}=AP$ is a closest matrix to
$A=A^{(d)}\otimes A^{(d-1)}\otimes\cdots\otimes A^{(1)}$ in ${\mathcal B}$ in the
Frobenius norm, where $A^{(i)}\in\R^{p_i\times n_i}$, $i=1,2,\ldots,d$.
\end{proposition}

\begin{proof}
The proof is a straightforward modification of the proof of Proposition \ref{prop1}.
\end{proof}

Let $\widetilde{L}^{(1)},\widetilde{L}^{(2)},\ldots,\widetilde{L}^{(d)}$ be a sequence
of square nonsingular matrices, and let $L^{(1)},L^{(2)},\ldots,L^{(d)}$ be
regularization matrices with desirable null spaces. It follows from Proposition
\ref{prop3} that a closest matrix to
\[
\widetilde{L}=\widetilde{L}^{(d)}\otimes\widetilde{L}^{(d-1)}\otimes\cdots\otimes
\widetilde{L}^{(1)}
\]
with null space ${\mathcal N}(L^{(d)}\otimes L^{(d-1)}\otimes\cdots\otimes L^{(1)})$ is
\[
L=\widetilde{L}^{(d)}P^{(d)}\otimes\widetilde{L}^{(d-1)}P^{(d-1)}\otimes\cdots\otimes
\widetilde{L}^{(1)}P^{(1)},
\]
where the orthogonal projectors $P^{(i)}$ are defined by \eqref{Pi} and the matrix
$V_{\ell_i}^{(i)}\in\R^{n_i\times\ell_i}$ has $1\leq\ell_i<n_i$ orthonormal columns that span
${\mathcal N}(L^{(i)})$ for $i=1,2,\ldots,d$.

The following result generalizes Corollary \ref{cor1} to higher space-dimensions and
to rectangular blocks of different sizes.
\begin{proposition}\label{prop4}
Let $V_{\ell_i}^{(i)}\in\R^{n_i\times\ell_i}$ have $1\leq\ell_i<n_i$ orthonormal columns for
$i=1,2,\ldots,d$, and let ${\mathcal B}$ denote the subspace of matrices of the form
\[
B=B^{(d)}\otimes B^{(d-1)}\otimes\cdots \otimes B^{(1)},
\]
where the range of $B^{(i)}\in\R^{p_i\times n_i}$  is orthogonal to
${\mathcal R}(V_{\ell_i}^{(i)})$ for  all $i$.
Let $P$  be defined by \eqref{Pi}. Then the matrix $\widehat{A}=PA$ is a closest matrix to
$A=A^{(d)}\otimes A^{(d-1)}\otimes\cdots\otimes A^{(1)}$ in ${\mathcal B}$ in the
Frobenius norm, where $A^{(i)}\in\R^{p_i\times n_i}$, $i=1,2,\ldots,d$.
\end{proposition}
\begin{proof}
The result can be shown by modifying the proof of Propositions \ref{prop1} or \ref{prop2}.
\end{proof}

Let $L^{(1)},L^{(2)},\ldots,L^{(d)}$ be  a sequence of
regularization matrices with desirable ranges, and let $\widetilde{L}^{(1)},\widetilde{L}^{(2)},\ldots,\widetilde{L}^{(d)}$ 
be  full rank matrices.  It follows from Proposition
\ref{prop4} that a closest matrix to
\[
\widetilde{L}=\widetilde{L}^{(d)}\otimes\widetilde{L}^{(d-1)}\otimes\cdots\otimes
\widetilde{L}^{(1)}
\]
with range in  ${\mathcal R}(L^{(d)}\otimes L^{(d-1)}\otimes\cdots\otimes L^{(1)})$ is
\[
L=P^{(d)}\widetilde{L}^{(d)}\otimes P^{(d-1)}\widetilde{L}^{(d-1)}\otimes\cdots\otimes
P^{(1)}\widetilde{L}^{(1)},
\]
where the orthogonal projectors $P^{(i)}$ are defined by \eqref{Pi} and the matrix
$V_{\ell_i}^{(i)}\in\R^{n_i\times\ell_i}$ has $1\leq\ell_i<n_i$ orthonormal columns that span
${\mathcal N}(L^{(i)})$ for $i=1,2,\ldots,d$.

We conclude this section with an extension of \eqref{blkeq:Ptik2} to higher
space-dimensions and assume that the problem has a nested tensor structure, i.e.,
\[
K^{(i)} = K^{(i,2)} \otimes K^{(i,1)},
\]
where $K^{(1,i)}\in\R^{n_i\times n_i}$, $K^{(2,i)}\in\R^{s_i\times s_i}$, $i=1,2$, and
that
\[
B=B^{(2)} \otimes B^{(1)},
\]
where $B^{(i)}\in\R^{n_i\times s_i}$ for $i=1,2$. The minimization problem \eqref{eq:sys}
with
\[
K=K^{(2,2)} \otimes K^{(2,1)}\otimes K^{(1,2)} \otimes K^{(1,1)}
\]
and $\bb=\mathrm{vec}(B)$ reads
\[
\min_{X\in\R^{n\times s}}\left\{\|(K^{(1,2)}\otimes K^{(1,1)})X(K^{(2,2)T}\otimes
K^{(2,1)T})-B^{(2)} \otimes B^{(1)}\|_F^2\right\}.
\]
Let the regularization matrices have a nested tensor structure
\[
L^{(i)} = L^{(i,2)} \otimes L^{(i,1)},\quad i=1,2.
\]
Then penalized least-squares problem that has to be solved is of the form
\begin{equation}\label{prb1}
\begin{array}{rcl}
\displaystyle{\min_{X\in\R^{n\times s}}}
\{\|(K^{(1,2)}\otimes K^{(1,1)})X(K^{(2,2)T}\otimes K^{(2,1)T})-B^{(2)}
\otimes B^{(1)}\|_F^2+\\ \mu\|(L^{(1,2)} \otimes L^{(1,1)}) X(L^{(2,2)T}
\otimes L^{(2,1)T}\|_F^2\}.
\end{array}
\end{equation}

If, moreover, the solution is separable of the form $X=X^{(2)} \otimes X^{(1)}$, where
$X^{(i)}\in\R^{n_i\times s_i}$ for $i=1,2$, then we obtain the minimization problem
\begin{equation}\label{prb2}
\begin{array}{rcl}
\displaystyle{
\min_{\substack{X^{(1)}\in\R^{n_1\times s_1}\\ X^{(2)}\in\R^{n_2\times s_2}}}}
\{\|(K^{(1,2)}X^{(2)}K^{(2,2)T})\otimes (K^{(1,1)}X^{(1)}K^{(2,1)T})-B^{(2)} \otimes
B^{(1)}\|_F^2+\\
\mu\|(L^{(1,2)}X^{(2)}L^{(2,2)T})\otimes (L^{(1,1)}X^{(1)}L^{(2,1)T})\|_F^2\}.
\end{array}
\end{equation}

When the regularization matrices are of the form
$L^{(i,j)}=P^{(i,j)}\widetilde{L}^{(i,j)}$, $1\leq i,j\leq 2$, where the $P^{(i,j)}$ are
orthogonal projectors and the $\widetilde{L}^{(i,j)}$ are square and invertible, the
minimization problems \eqref{prb1} and \eqref{prb2} can be transformed similarly as
equation \eqref{blkeq:tik1} was transformed into \eqref{blkeq:Ptik2}.

\section{Computed examples}\label{sec4}
We illustrate the performance of regularization matrices of the form 
$L = L^{(2)} \otimes L^{(1)}$ with $L^{(i)}=P^{(i)}\widetilde{L}^{(i)}$ or 
$L^{(i)}=\widetilde{L}^{(i)}P^{(i)}$ for $i = 1,2$, and compare with the regularization
matrices $L^{(i)}=\widetilde{L}^{(i)}$ for $i = 1,2$. The noise level is given by 
\[
\nu:=\frac{{\|E\|_F}}{\|\widehat{B}\|_F}.
\]
Here $E=B-\widehat{B}$ is the error matrix, where $B$ is the available error-contaminated 
matrix in (\ref{blkeq:tik1}) and $\widehat{B}$ is the associated unknown error-free matrix,
i.e., $\widehat{\bb}={\rm vec}(\widehat{B})$; cf. (\ref{consistent}). In all examples, the 
entries of the matrix $E$ are normally distributed with zero mean and are scaled to 
correspond to a specified noise level. We let $\eta=1.01$ in (\ref{nonlin}) in all 
examples. The quality of computed approximate solutions $X_{\mu.k}$ of (\ref{blkeq:tik1}) 
is measured with the relative error norm
\[
e_k:=\frac{{\|X_{\mu,k}-\widehat{X}\|_F}}{\|\widehat{X}\|_F}.
\]
All computations were carried out in MATLAB R2017a with about $15$ significant decimal 
digits on a laptop computer with an Intel Core i7-6700HQ CPU @ 2.60GHz processor and 16GB 
RAM.

\begin{table}[tbh]
\begin{center}
\begin{tabular}{cccc}
 \hline
regularization & number of & CPU time & relative \\
matrix & iterations $k$ & in seconds & error $e_k$ \\
\hline
&& $\nu=1\cdot10^{-3}$&\\
\hline
$\widetilde{L}^{(1)} \otimes \widetilde{L}^{(1)}$                 & $1$ & $11.9 $ & $8.42 \cdot 10^{-2}$ \\
$P^{(1)}\widetilde{L}^{(1)} \otimes P^{(1)}\widetilde{L}^{(1)}$   & $1$ & $11.6 $ & $6.17 \cdot 10^{-2}$ \\
$\widetilde{L}^{(1)}P^{(1)} \otimes \widetilde{L}^{(1)}P^{(1)}$   & $2$ & $11.6 $ & $6.85 \cdot 10^{-2}$ \\
$\widetilde{L}^{(2)} \otimes \widetilde{L}^{(1)}$                 & $1$ & $12.3 $ & $9.69 \cdot 10^{-2}$ \\
$P^{(2)}\widetilde{L}^{(2)} \otimes P^{(1)}\widetilde{L}^{(1)}$   & $1$ & $11.7 $ & $6.17 \cdot 10^{-2}$ \\
$\widetilde{L}^{(2)}P^{(2)} \otimes \widetilde{L}^{(1)}P^{(1)}$   & $1$ & $11.9 $ & $7.18 \cdot 10^{-2}$ \\
$\widetilde{L}^{(2)} \otimes \widetilde{L}^{(2)}$                 & $1$ & $12.1 $ & $1.05 \cdot 10^{-1}$ \\
$P^{(2)}\widetilde{L}^{(2)} \otimes P^{(2)}\widetilde{L}^{(2)}$   & $1$ & $12.1 $ & $6.17 \cdot 10^{-2}$ \\
$\widetilde{L}^{(2)}P^{(2)} \otimes \widetilde{L}^{(2)}P^{(2)}$   & $1$ & $11.8 $ & $8.55 \cdot 10^{-2}$ \\
\hline
&&noise level $\nu=1\cdot10^{-4}$&\\
\hline
$\widetilde{L}^{(1)} \otimes \widetilde{L}^{(1)}$                 & $8$ & $11.7 $ & $4.98 \cdot 10^{-2}$ \\
$P^{(1)}\widetilde{L}^{(1)} \otimes P^{(1)}\widetilde{L}^{(1)}$   & $1$ & $11.6 $ & $4.72 \cdot 10^{-2}$ \\
$\widetilde{L}^{(1)}P^{(1)} \otimes \widetilde{L}^{(1)}P^{(1)}$   & $8$ & $11.5 $ & $4.80 \cdot 10^{-2}$ \\
$\widetilde{L}^{(2)} \otimes \widetilde{L}^{(1)}$                 & $6$ & $12.1 $ & $5.14 \cdot 10^{-2}$ \\
$P^{(2)}\widetilde{L}^{(2)} \otimes P^{(1)}\widetilde{L}^{(1)}$   & $1$ & $11.7 $ & $4.72 \cdot 10^{-2}$ \\
$\widetilde{L}^{(2)}P^{(2)} \otimes \widetilde{L}^{(1)}P^{(1)}$   & $7$ & $11.9 $ & $4.84 \cdot 10^{-2}$ \\
$\widetilde{L}^{(2)} \otimes \widetilde{L}^{(2)}$                 & $5$ & $11.8 $ & $4.96 \cdot 10^{-2}$ \\
$P^{(2)}\widetilde{L}^{(2)} \otimes P^{(2)}\widetilde{L}^{(2)}$   & $1$ & $11.6 $ & $4.72 \cdot 10^{-2}$ \\
$\widetilde{L}^{(2)}P^{(2)} \otimes \widetilde{L}^{(2)}P^{(2)}$   & $6$ & $11.7 $ & $4.81 \cdot 10^{-2}$ \\
\hline
\end{tabular}
\end{center}
\caption{Example 4.1: Number of iterations, CPU time in seconds, and relative error $e_k$ 
in computed approximate solutions $X_{\mu,k}$ determined by Tikhonov regularization based 
on the global Arnoldi process for two noise levels and several regularization matrices.}\label{tab4.1}
\end{table}

\begin{figure}[h!tbp]
\centering
\includegraphics[height=5cm,width=7cm]{shaw150_RecoveredP1L1P1L110-3.eps}
\caption{Example 4.1: Computed approximate solution ${X}_{\mu,1}$ for noise level 
$\nu=1\cdot10^{-3}$ and regularization matrix 
$P^{(1)}\widetilde{L}^{(1)} \otimes P^{(1)}\widetilde{L}^{(1)}$ using the discrepancy 
principle.}\label{fig4.1}
\end{figure}

Example 4.1. Consider the Fredholm integral equation of the first kind in two 
space-dimensions,
\begin{equation}\label{shaw}
\int\int_{\Omega}\kappa(\tau,\sigma;x,y)f(x,y)dxdy=g(\tau,\sigma),\quad 
(\tau,\sigma)\in \Omega,
\end{equation}
where $\Omega=[-\pi/2,\pi/2]\times[-\pi/2,\pi/2]$. The kernel is given by
\[
\kappa_1(\tau,\sigma;x,y)=\kappa_1(\tau,x)\kappa_1(\sigma,y),\quad (\tau,\sigma),(x,y)
\in \Omega,
\]
where 
\[
\kappa(\tau,\sigma)=(\cos(\sigma)+\sin(\tau))^{2}\left(\frac{\sin(\xi)}{\xi}\right)^{2},\quad
\xi=\pi(\sin(\sigma)+\cos(\tau)).
\]

The right-hand side function is of the form
\[
g(\tau,\sigma)=h(\tau)h(\sigma),
\]
where $h(\sigma)$ is chosen so that the solution is the sum of two Gaussian functions and 
a constant. We use the MATLAB code {\sf shaw} from \cite{Hansen} to discretize (\ref{shaw}) 
by a Galerkin method with $150\times 150$ orthonormal box functions as test and trial 
functions. This code produces the matrix $K\in{\R}^{150\times 150}$ that approximates the
analogue of the integral operator \eqref{shaw} in one space-dimension, and a discrete 
approximate solution $\bx_1$ in one space-dimension. Adding the vector 
${\bn}_1=[1,1,\ldots,1]^T$ yields the vector $\widehat{\bx}_1\in{\R}^{150}$, from which we
construct the scaled discrete approximation $\widehat{X}=\widehat{\bx}_1\widehat{\bx}_1^T$ 
of the solution of (\ref{shaw}). The error-free right-hand side is computed by 
$\widehat{B}=K\widehat{X}K^T$. The error matrix $E\in{\R}^{150\times 150}$ models white 
Gaussian noise with noise levels $\nu = 1\cdot 10^{-3}$ and $\nu = 1\cdot 10^{-4}$. The 
data matrix $B$ in (\ref{blkeq:tik1}) is computed as $B=\widehat{B}+E$. The regularization 
matrices $L$ used are constructed like in Examples 2.1-2.4. We compare the performance of 
these regularization matrices to the performance of the nonsingular regularization matrices 
$L=\widetilde{L}^{(i)} \otimes \widetilde{L}^{(i)}$, $i=1,2$, and 
$L=\widetilde{L}^{(2)} \otimes \widetilde{L}^{(1)}$. The number of steps, $k$, of the 
global Arnoldi method is chosen as small as possible so that the discrepancy principle
(\ref{nonlin}) can be satisfied. The regularization parameter is determined by the 
discrepancy principle.

Table \ref{tab4.1} displays results obtained for the different regularization matrices and
noise levels. The table shows the regularization matrices
$P^{(i)}\widetilde{L}^{(i)} \otimes P^{(i)}\widetilde{L}^{(i)}$, $i=1,2$, as well as
$P^{(2)}\widetilde{L}^{(2)} \otimes P^{(1)}\widetilde{L}^{(1)}$ to yield the smallest 
relative errors. Moreover, the computation with these regularization matrices requires the 
least CPU time. Figure \ref{fig4.1} shows the computed approximate solution for the
noise level $\nu=1\cdot 10^{-3}$ when the regularization matrix 
$P^{(1)}\widetilde{L}^{(1)} \otimes P^{(1)}\widetilde{L}^{(1)}$ is used. The computed 
approximation cannot be visually distinguished from the desired exact solution 
$\widehat{X}$. We therefore do not show the latter. $~~~\Box$

\begin{table}[tbh]
\begin{center}
\begin{tabular}{cccc}
\hline
regularization & number of & CPU time & relative \\
matrix & iterations $k$ & in seconds & error $e_k$ \\
\hline
&&noise level $\nu=1\cdot10^{-2}$&\\
\hline
$\widetilde{L}^{(1)} \otimes \widetilde{L}^{(1)}$                 & $11$ & $18.9 $ & $7.90 \cdot 10^{-2}$ \\
$P^{(1)}\widetilde{L}^{(1)} \otimes P^{(1)}\widetilde{L}^{(1)}$   & $1$ & $19.6 $ &  $6.36 \cdot 10^{-2}$ \\
$\widetilde{L}^{(1)}P^{(1)} \otimes \widetilde{L}^{(1)}P^{(1)}$   & $11$ & $19.3 $ & $7.90 \cdot 10^{-2}$ \\
$\widetilde{L}^{(2)} \otimes \widetilde{L}^{(1)}$                 & $10$ & $19.5 $ & $8.39 \cdot 10^{-2}$ \\
$P^{(2)}\widetilde{L}^{(2)} \otimes P^{(1)}\widetilde{L}^{(1)}$   & $1$ & $18.9 $ &  $7.81 \cdot 10^{-2}$ \\
$\widetilde{L}^{(2)}P^{(2)} \otimes \widetilde{L}^{(1)}P^{(1)}$   & $10$ & $19.7 $ & $8.39 \cdot 10^{-2}$ \\
$\widetilde{L}^{(2)} \otimes \widetilde{L}^{(2)}$                 & $9$ & $19.7 $ &  $8.64 \cdot 10^{-2}$ \\
$P^{(2)}\widetilde{L}^{(2)} \otimes P^{(2)}\widetilde{L}^{(2)}$   & $1$ & $19.8 $ &  $7.81 \cdot 10^{-2}$ \\
$\widetilde{L}^{(2)}P^{(2)} \otimes \widetilde{L}^{(2)}P^{(2)}$   & $9$ & $19.0 $ &  $8.64 \cdot 10^{-2}$ \\
\hline
&&noise level $\nu=1\cdot10^{-3}$&\\
\hline
$\widetilde{L}^{(1)} \otimes \widetilde{L}^{(1)}$                 & $17$ & $22.0 $ & $1.58 \cdot 10^{-2}$ \\
$P^{(1)}\widetilde{L}^{(1)} \otimes P^{(1)}\widetilde{L}^{(1)}$   & $1$ & $22.0 $ &  $9.95 \cdot 10^{-3}$ \\
$\widetilde{L}^{(1)}P^{(1)} \otimes \widetilde{L}^{(1)}P^{(1)}$   & $17$ & $22.1 $ & $1.58 \cdot 10^{-2}$ \\
$\widetilde{L}^{(2)} \otimes \widetilde{L}^{(1)}$                 & $17$ & $23.5 $ & $1.53 \cdot 10^{-2}$ \\
$P^{(2)}\widetilde{L}^{(2)} \otimes P^{(1)}\widetilde{L}^{(1)}$   & $1$ & $22.0 $ &  $9.79 \cdot 10^{-3}$ \\
$\widetilde{L}^{(2)}P^{(2)} \otimes \widetilde{L}^{(1)}P^{(1)}$   & $17$ & $22.4 $ & $1.53 \cdot 10^{-2}$ \\
$\widetilde{L}^{(2)} \otimes \widetilde{L}^{(2)}$                 & $16$ & $22.5 $ & $2.04 \cdot 10^{-2}$ \\
$P^{(2)}\widetilde{L}^{(2)} \otimes P^{(2)}\widetilde{L}^{(2)}$   & $1$ & $21.5 $ &  $9.79 \cdot 10^{-3}$ \\
$\widetilde{L}^{(2)}P^{(2)} \otimes \widetilde{L}^{(2)}P^{(2)}$   & $16$ & $21.9 $ & $2.04 \cdot 10^{-2}$ \\
\hline
&&noise level $\nu=1\cdot10^{-4}$&\\
\hline
$\widetilde{L}^{(1)} \otimes \widetilde{L}^{(1)}$                 & $21$ & $44.2 $ & $2.38 \cdot 10^{-3}$ \\
$P^{(1)}\widetilde{L}^{(1)} \otimes P^{(1)}\widetilde{L}^{(1)}$   & $1$ & $44.1 $ &  $1.91 \cdot 10^{-3}$ \\
$\widetilde{L}^{(1)}P^{(1)} \otimes \widetilde{L}^{(1)}P^{(1)}$   & $21$ & $44.1 $ & $2.38 \cdot 10^{-3}$ \\
$\widetilde{L}^{(2)} \otimes \widetilde{L}^{(1)}$                 & $21$ & $45.2 $ & $2.35 \cdot 10^{-3}$ \\
$P^{(2)}\widetilde{L}^{(2)} \otimes P^{(1)}\widetilde{L}^{(1)}$   & $1$ & $44.9 $ &  $1.91 \cdot 10^{-3}$ \\
$\widetilde{L}^{(2)}P^{(2)} \otimes \widetilde{L}^{(1)}P^{(1)}$   & $21$ & $46.8 $ & $2.35 \cdot 10^{-3}$ \\
$\widetilde{L}^{(2)} \otimes \widetilde{L}^{(2)}$                 & $21$ & $45.4 $ & $2.33 \cdot 10^{-3}$ \\
$P^{(2)}\widetilde{L}^{(2)} \otimes P^{(2)}\widetilde{L}^{(2)}$   & $1$ & $44.7 $ &  $1.91 \cdot 10^{-3}$ \\
$\widetilde{L}^{(2)}P^{(2)} \otimes \widetilde{L}^{(2)}P^{(2)}$   & $21$ & $44.9 $ & $2.33 \cdot 10^{-3}$ \\
\hline
\end{tabular}
\end{center}
\caption{Example 4.2: Number of iterations, CPU time in seconds, and relative error $e_k$ 
in computed approximate solutions $X_{\mu,k}$ determined by Tikhonov regularization based 
on the global Arnoldi process for two noise levels and several regularization matrices.}\label{tab4.3}
\end{table}

\begin{figure}[h!tbp]
\centering
\subfigure[]{\includegraphics[height=5cm,width=5cm]{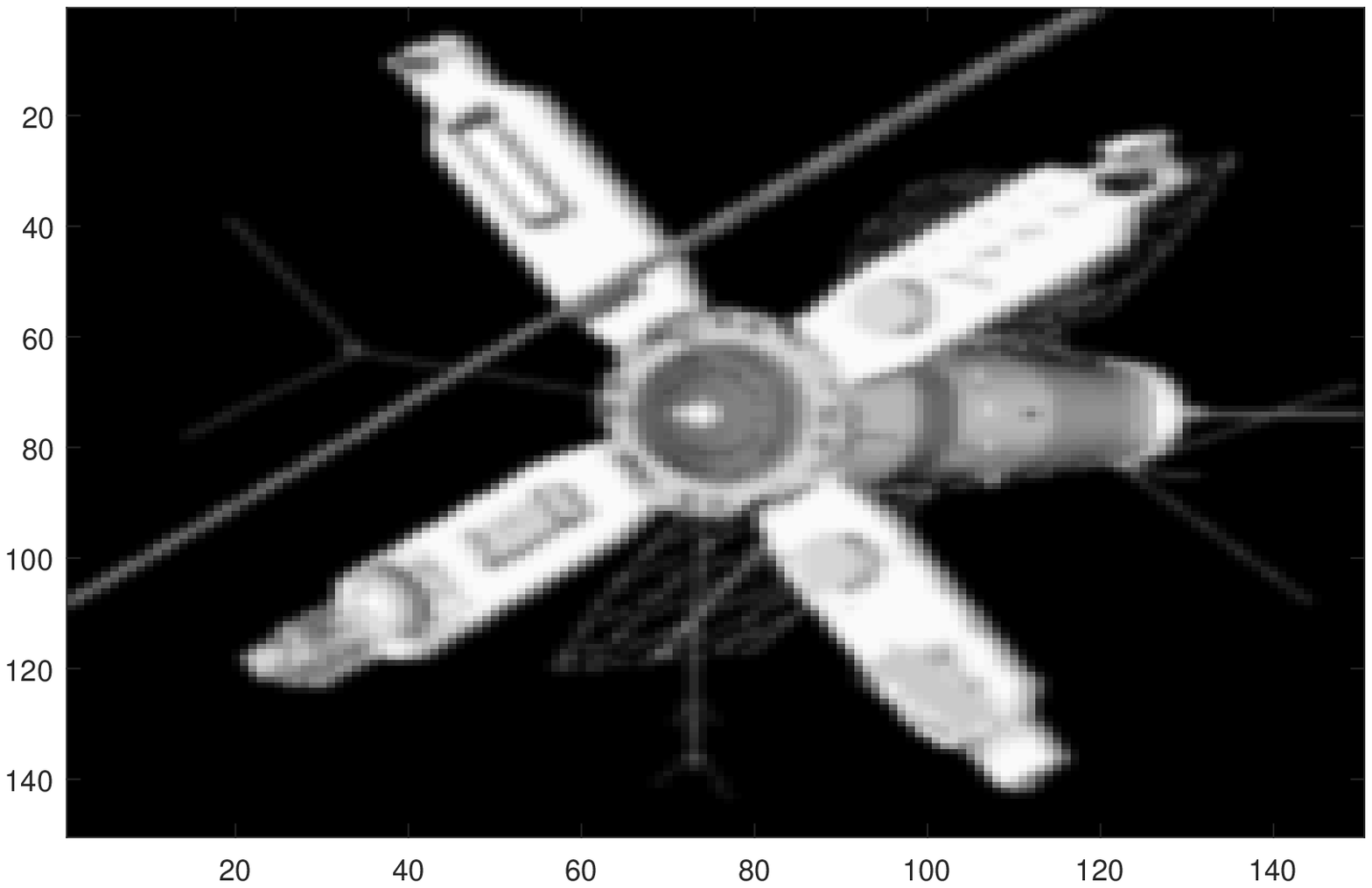}}
\subfigure[]{\includegraphics[height=5cm,width=5cm]{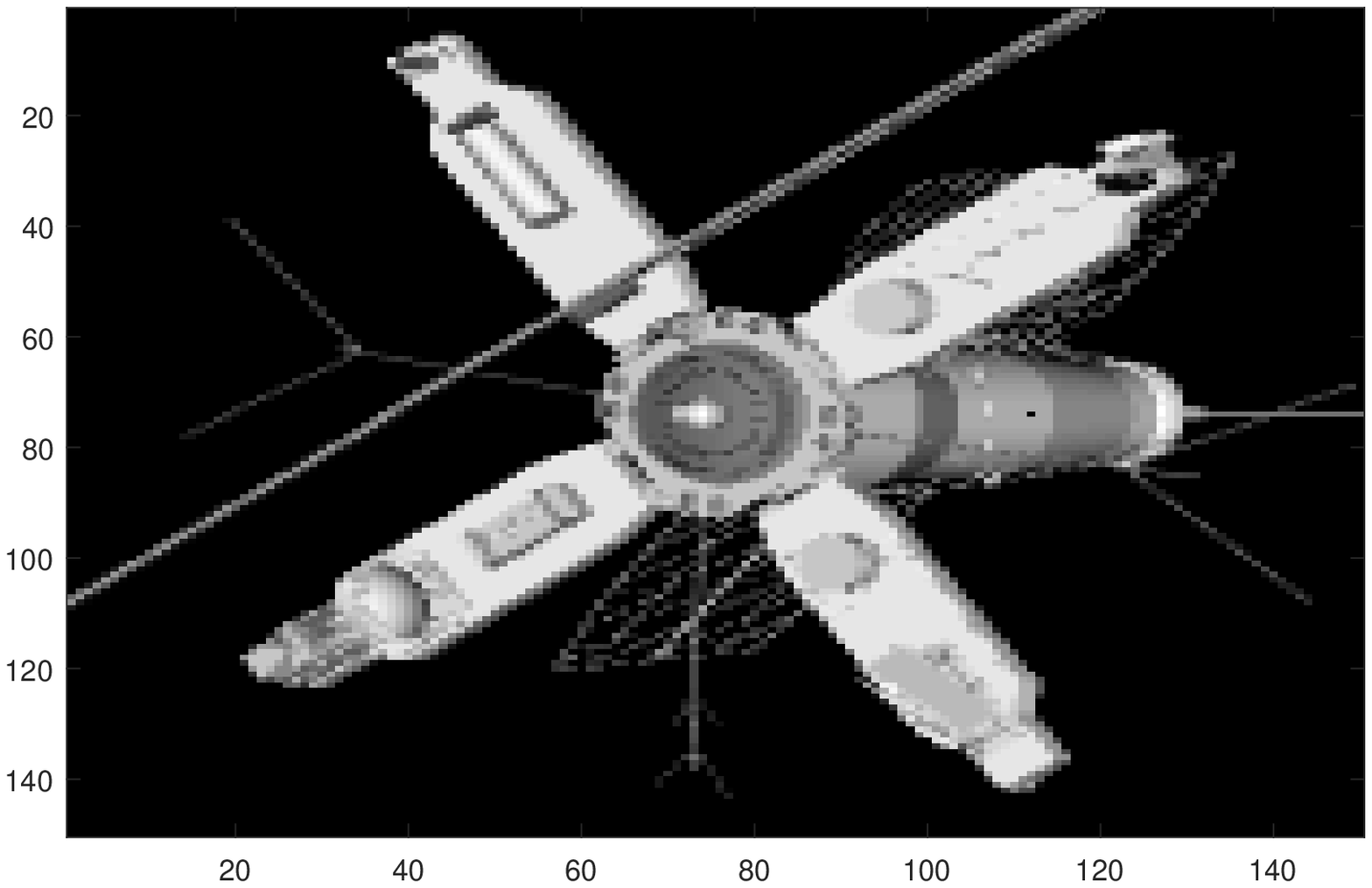}}
\subfigure[]{\includegraphics[height=5cm,width=5cm]{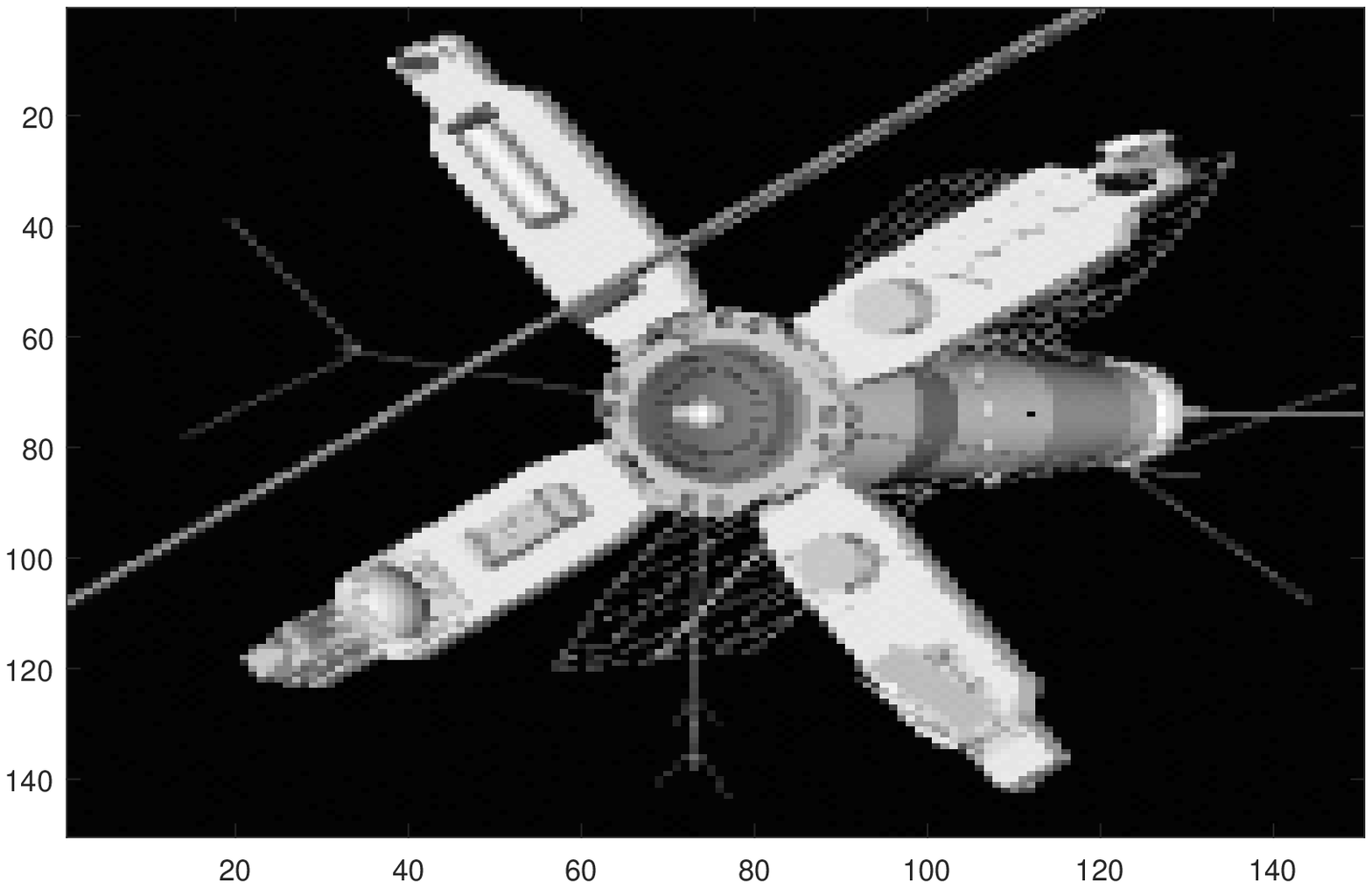}}
\subfigure[]{\includegraphics[height=5cm,width=5cm]{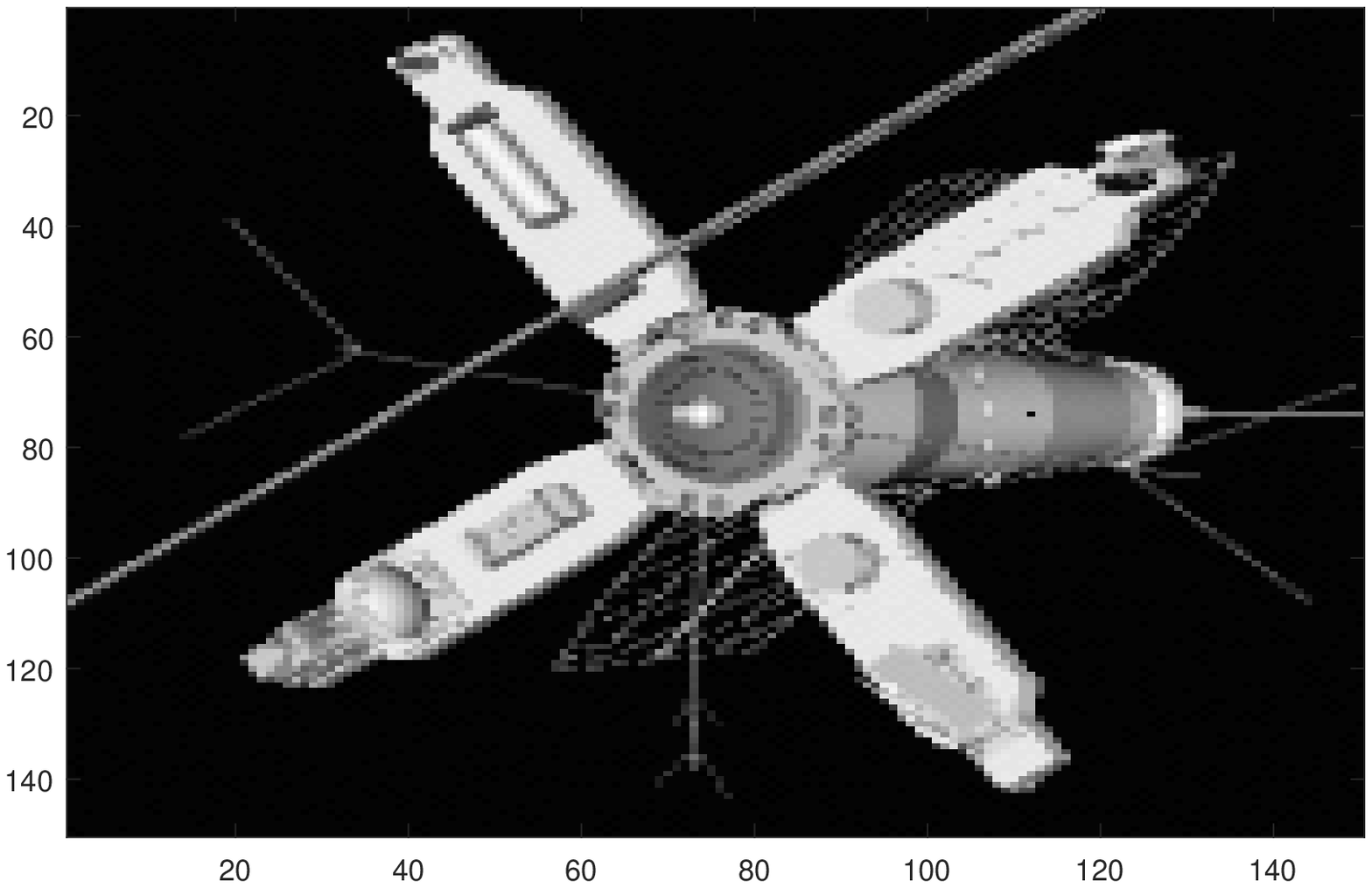}}
\caption{Example 4.2: (a) Available blur- and noise-contaminated \texttt{satellite} image 
represented by the matrix $B$, (b) desired image, (c) restored image for the noise level 
$\nu=1\cdot10^{-3}$ and regularization matrix 
$P^{(1)}\widetilde{L}^{(1)} \otimes P^{(1)}\widetilde{L}^{(1)}$, and (d) restored image 
for the same noise level and regularization matrix 
$P^{(2)}\widetilde{L}^{(2)} \otimes P^{(2)}\widetilde{L}^{(2)}$.}
\label{fig4.3}
\end{figure}

Example 4.2. We consider the restoration of the test image \texttt{satellite}, which is 
represented by an array of $150\times 150$ pixels. The available image, represented by the
matrix $B\in\R^{150\times 150}$, is corrupted by Gaussian blur and additive zero-mean white
Gaussian noise; it is shown in Figure \ref{fig4.3}(a). Figure \ref{fig4.3}(b) displays the 
desired blur- and noise-free image. It is represented by the matrix 
$\widehat{X}\in \R^{150\times 150}$, and is assumed not to be known. The blurring matrices
$K^{(i)}\in \R^{150\times 150}$, $i=1,2$, are Toeplitz matrices. We let 
$K^{(1)}=K^{(2)}=K$, where $K$ is analogous to the matrix generated by the MATLAB function 
{\sf blur} from \cite{Hansen} using the parameter values $\texttt{band}=5$ and 
$\texttt{sigma}=1.5$. We show results for the noise levels $\nu = 1\cdot 10^{-j}$, 
$j=2,3,4$. The data matrix $B$ in (\ref{blkeq:tik1}) is determined similarly as in Example
4.1 and the regularization matrices used are the same as in Example 4.1. 

Table \ref{tab4.3} displays the number of iterations, CPU time, and the relative errors 
$e_k$ in the computed approximate solutions $X_{\mu,k}$ determined by the global Arnoldi 
process with data matrices contaminated by noise of levels $\nu=1\cdot10^{-j}$, $j=2,3,4$,
for several regularization matrices. The iterations are terminated as soon as the 
discrepancy principle can be satisfied and the regularization parameter then is chosen so
that (\ref{nonlin}) holds. Table \ref{tab4.3} shows the global Arnoldi process with the 
regularization matrices $P^{(i)}\widetilde{L}^{(i)} \otimes P^{(i)}\widetilde{L}^{(i)}$,
$i =1,2$, and $P^{(2)}\widetilde{L}^{(2)} \otimes P^{(1)}\widetilde{L}^{(1)}$ to yield the
best approximations of $\widehat{X}$ and to require the least CPU time. Figures 
\ref{fig4.3}(c) and \ref{fig4.3}(d) show the computed approximate solutions determined by 
the global Arnoldi process with $\nu=1\cdot10^{-3}$ and the regularization matrices 
$P^{(1)}\widetilde{L}^{(1)} \otimes P^{(1)}\widetilde{L}^{(1)}$ and 
$P^{(2)}\widetilde{L}^{(2)} \otimes P^{(2)}\widetilde{L}^{(2)}$, respectively. The quality
of the computed restorations is visually indistinguishable. $~~~\Box$

\begin{table}[tbh]
\begin{center}
\begin{tabular}{cccc}
\hline
regularization & number of & CPU time & relative \\
matrix & iterations $k$ & in seconds & error $e_k$ \\
\hline
&&noise level $\nu=1\cdot10^{-3}$&\\
\hline
$\widetilde{L}^{(1)} \otimes \widetilde{L}^{(1)}$                 & $14$ & $22.7 $ & $1.22 \cdot 10^{-2}$ \\
$P^{(1)}\widetilde{L}^{(1)} \otimes P^{(1)}\widetilde{L}^{(1)}$   & $1$ & $19.8 $ &  $7.47 \cdot 10^{-3}$ \\
$\widetilde{L}^{(1)}P^{(1)} \otimes \widetilde{L}^{(1)}P^{(1)}$   & $14$ & $20.4 $ & $1.22 \cdot 10^{-2}$ \\
$\widetilde{L}^{(2)} \otimes \widetilde{L}^{(1)}$                 & $14$ & $21.2 $ & $1.15 \cdot 10^{-2}$ \\
$P^{(2)}\widetilde{L}^{(2)} \otimes P^{(1)}\widetilde{L}^{(1)}$   & $1$ & $21.2 $ &  $7.60 \cdot 10^{-3}$ \\
$\widetilde{L}^{(2)}P^{(2)} \otimes \widetilde{L}^{(1)}P^{(1)}$   & $14$ & $20.5 $ & $1.15 \cdot 10^{-2}$ \\
$\widetilde{L}^{(2)} \otimes \widetilde{L}^{(2)}$                 & $13$ & $20.9 $ & $1.35 \cdot 10^{-2}$ \\
$P^{(2)}\widetilde{L}^{(2)} \otimes P^{(2)}\widetilde{L}^{(2)}$   & $1$ & $21.1 $ &  $7.60 \cdot 10^{-3}$ \\
$\widetilde{L}^{(2)}P^{(2)} \otimes \widetilde{L}^{(2)}P^{(2)}$   & $13$ & $20.9 $ & $1.35 \cdot 10^{-2}$ \\
\hline
&&noise level $\nu=1\cdot10^{-4}$&\\
\hline
$\widetilde{L}^{(1)} \otimes \widetilde{L}^{(1)}$                 & $20$ & $19.0 $ & $2.20 \cdot 10^{-3}$ \\
$P^{(1)}\widetilde{L}^{(1)} \otimes P^{(1)}\widetilde{L}^{(1)}$   & $1$ & $18.7 $ &  $2.05 \cdot 10^{-3}$ \\
$\widetilde{L}^{(1)}P^{(1)} \otimes \widetilde{L}^{(1)}P^{(1)}$   & $20$ & $18.7 $ & $2.20 \cdot 10^{-3}$ \\
$\widetilde{L}^{(2)} \otimes \widetilde{L}^{(1)}$                 & $20$ & $19.1 $ & $2.19 \cdot 10^{-3}$ \\
$P^{(2)}\widetilde{L}^{(2)} \otimes P^{(1)}\widetilde{L}^{(1)}$   & $1$ & $19.0 $ &  $2.04 \cdot 10^{-3}$ \\
$\widetilde{L}^{(2)}P^{(2)} \otimes \widetilde{L}^{(1)}P^{(1)}$   & $20$ & $18.7 $ & $2.19 \cdot 10^{-3}$ \\
$\widetilde{L}^{(2)} \otimes \widetilde{L}^{(2)}$                 & $20$ & $19.0 $ & $2.18 \cdot 10^{-3}$ \\
$P^{(2)}\widetilde{L}^{(2)} \otimes P^{(2)}\widetilde{L}^{(2)}$   & $1$ & $18.6 $ &  $2.04 \cdot 10^{-3}$ \\
$\widetilde{L}^{(2)}P^{(2)} \otimes \widetilde{L}^{(2)}P^{(2)}$   & $20$ & $18.5 $ & $2.18 \cdot 10^{-3}$ \\
\hline
\end{tabular}
\end{center}
\caption{Example 4.3: Number of iterations, CPU time in seconds, and relative error $e_k$ 
in computed approximate solutions $X_{\mu,k}$ determined by Tikhonov regularization based 
on the global Arnoldi process for two noise levels and several regularization matrices.}\label{tab4.4}
\end{table}

\begin{figure}[h!tbp]
\centering
\subfigure[]{\includegraphics[height=5cm,width=5cm]{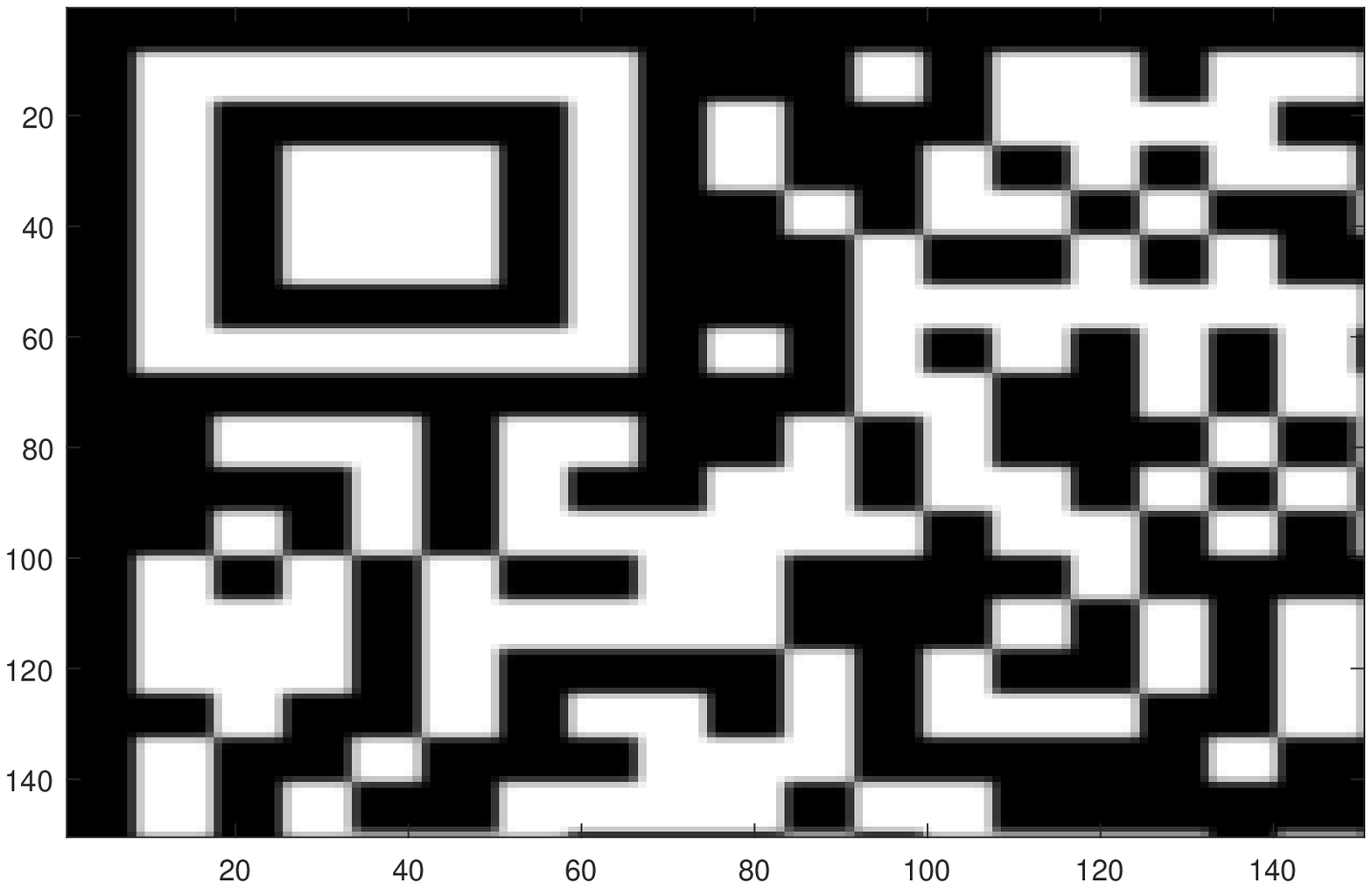}}
\subfigure[]{\includegraphics[height=5cm,width=5cm]{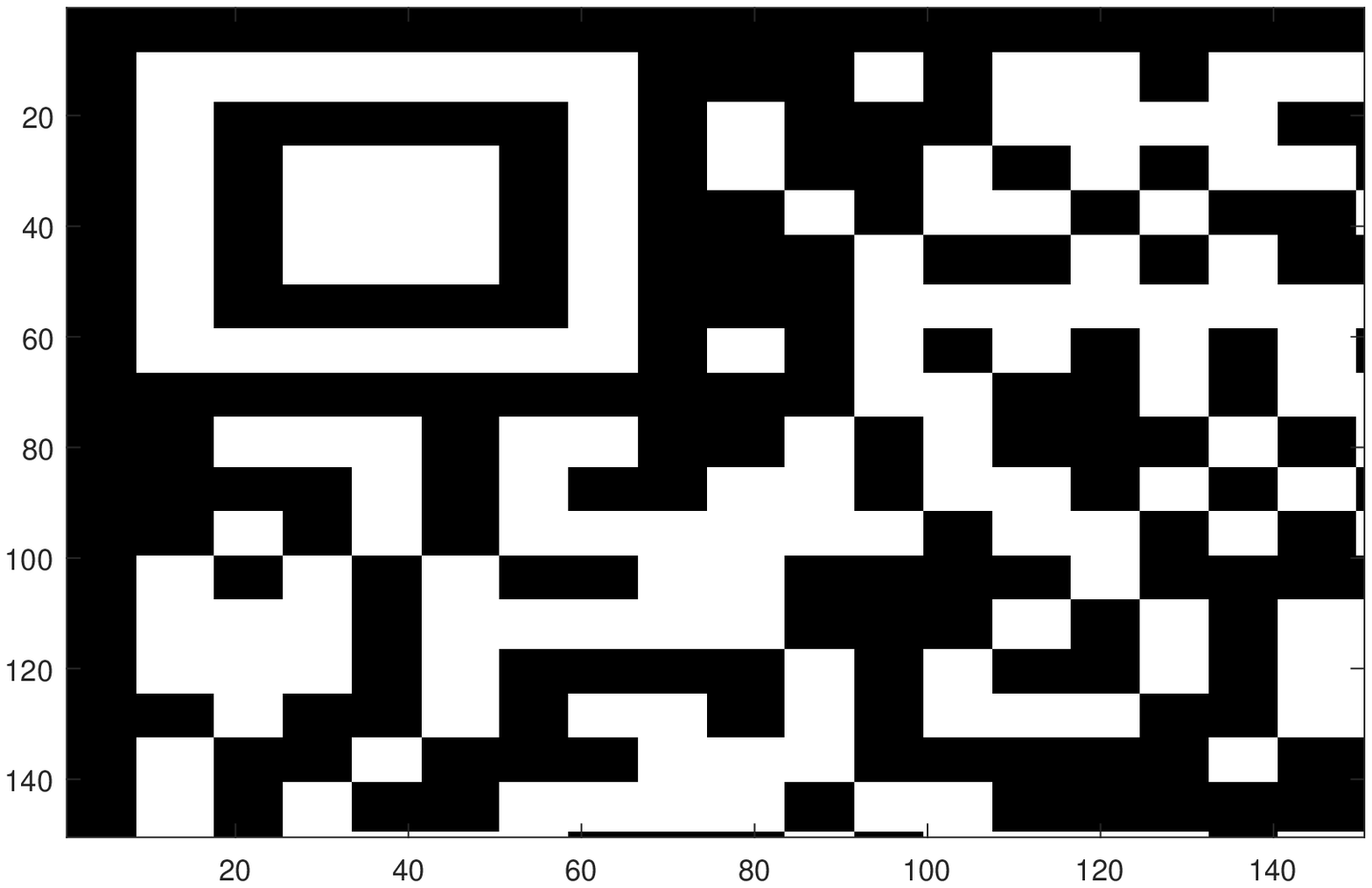}}
\subfigure[]{\includegraphics[height=5cm,width=5cm]{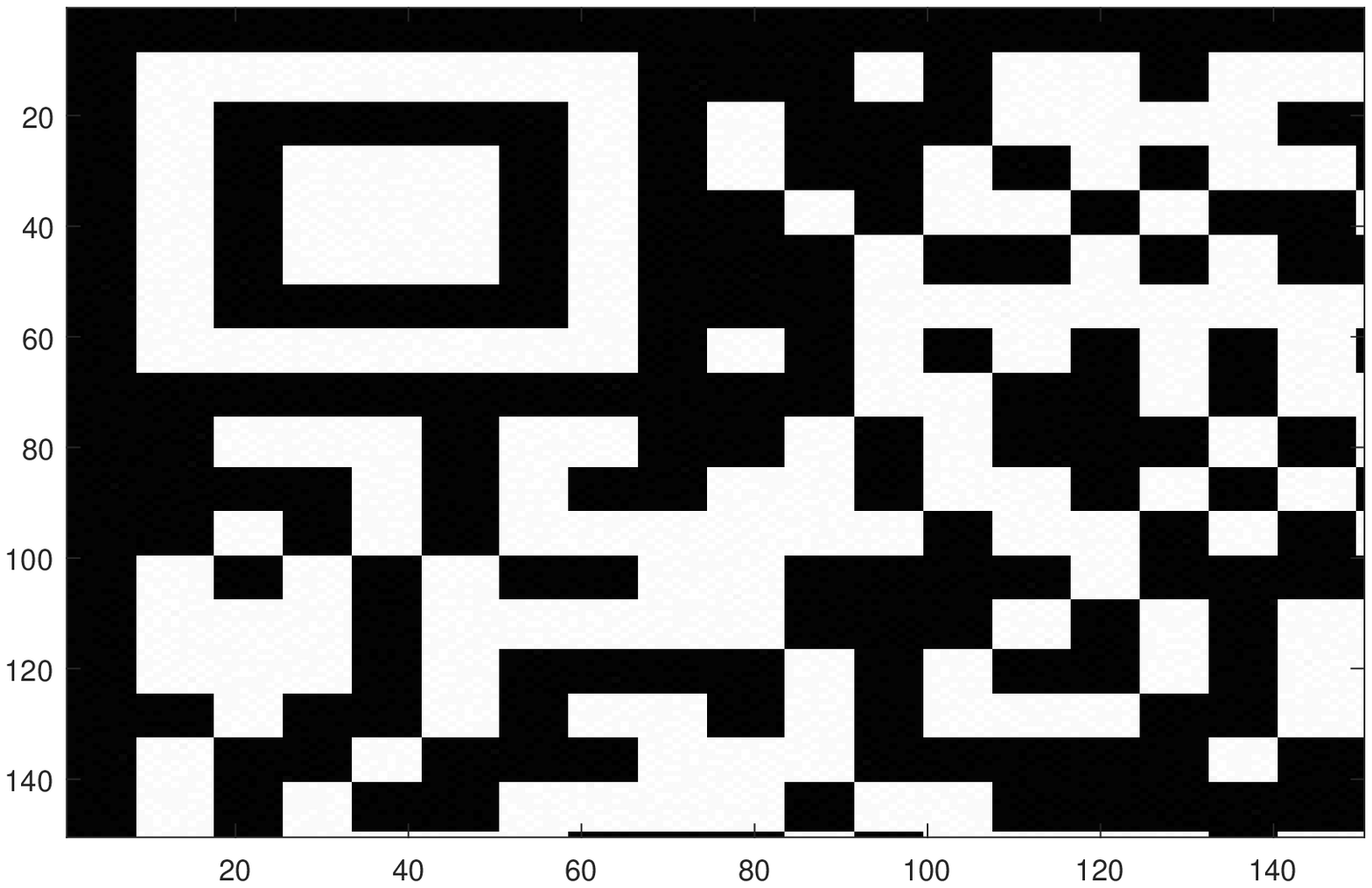}}
\subfigure[]{\includegraphics[height=5cm,width=5cm]{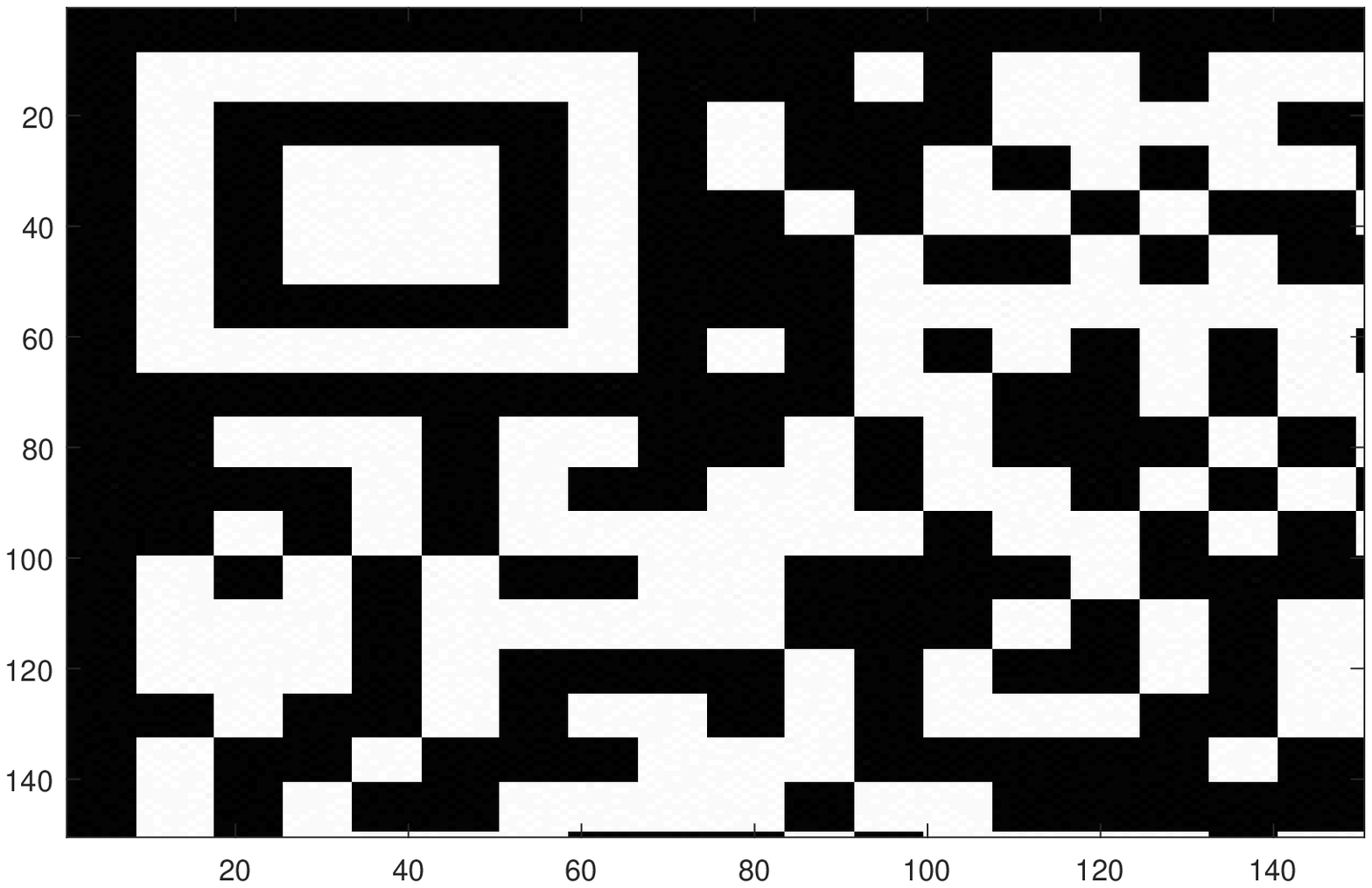}}
\caption{Example 4.3: (a) Available blur- and noise-contaminated \texttt{QR code} image 
represented by the matrix $B$, (b) desired image, (c) restored image for the noise level 
$\nu=1\cdot10^{-3}$ and regularization matrix 
$P^{(1)}\widetilde{L}^{(1)} \otimes P^{(1)}\widetilde{L}^{(1)}$, and (d) restored image 
for the same noise level and regularization matrix 
$P^{(2)}\widetilde{L}^{(2)} \otimes P^{(2)}\widetilde{L}^{(2)}$.}\label{fig4.4}
\end{figure}

{\bf Example 4.3.}
This example is similar to the previous one; only the image to be restored differs. Here 
we consider the restoration of the test image \texttt{QRcode}, which is represented by an
array of $150\times 150$ pixels corrupted by Gaussian blur, and additive zero-mean white 
Gaussian noise. Figure \ref{fig4.4}(a) shows the corrupted image that we would like to 
restore. It is represented by the matrix $B\in \R^{150\times 150}$. The desired blur- and
noise-free image is depicted in Figure \ref{fig4.4}(b). The blurring matrices
$K^{(i)}\in \R^{150\times 150},i=1,2,$ are Toeplitz matrices. They are generated like in 
Example 4.2. The regularization matrices $L$ are the same as in Example 4.2.

Table \ref{tab4.4} is analogous to Table \ref{tab4.3}. The iterations are terminated as 
soon as the discrepancy principle (\ref{nonlin}) can be satisfied. The table shows the
regularization matrices $P^{(i)}\widetilde{L}^{(i)} \otimes P^{(i)}\widetilde{L}^{(i)}$,
$i=1,2$, and $P^{(2)}\widetilde{L}^{(2)} \otimes P^{(1)}\widetilde{L}^{(1)}$ to yield the
most accurate approximations of $\widehat{X}$. Figures \ref{fig4.4}(c) and  
\ref{fig4.4}(d) show the restorations determined for $\nu=1\cdot10^{-3}$ with the 
regularization matrices $P^{(1)}\widetilde{L}^{(1)} \otimes P^{(1)}\widetilde{L}^{(1)}$ 
and $P^{(2)}\widetilde{L}^{(2)} \otimes P^{(2)}\widetilde{L}^{(2)}$, respectively. One 
cannot visually distinguish the quality of these restorations.  $~~~\Box$

\section{Concluding remarks}\label{sec5}
This paper presents a novel method to determine regularization matrices for discrete 
ill-posed problems in several space-dimensions by solving a matrix nearness problem. 
Numerical examples illustrate the effectiveness of the regularization matrices 
determined in this manner.  While all examples use the discrepancy principle to
determine a suitable value of the regularization parameter, other parameter choice rules 
also can be applied; see, e.g., \cite{BJ,RR} for discussions and references.

\section*{Acknowledgment}
Research by GH was supported in part by the Fund of Application Foundation of Science and 
Technology Department of the Sichuan Province (2016JY0249) and NNSF (11501392).
Research by SN was partially supported by SAPIENZA Universit\`a di Roma and INdAM-GNCS.

\end{document}